\pdfoutput=1
\RequirePackage{ifpdf}
\ifpdf 
\documentclass[pdftex]{sigma}
\else
\documentclass{sigma}
\fi

\begin{document}
\allowdisplaybreaks

\newcommand{\arXivNumber}{1905.08655}

\renewcommand{\PaperNumber}{081}

\FirstPageHeading

\ShortArticleName{A Note on the Derivatives of Isotropic Positive Definite Functions on the Hilbert Sphere}

\ArticleName{A Note on the Derivatives of Isotropic Positive\\ Definite Functions on the Hilbert Sphere}

\Author{Janin J\"AGER}

\AuthorNameForHeading{J.~J\"ager}

\Address{Lehrstuhl Numerische Mathematik, Justus-Liebig University,\\
Heinrich-Buff Ring 44, 35392 Giessen, Germany}
\Email{\href{mailto:janin.jaeger@math.uni-giessen.de}{janin.jaeger@math.uni-giessen.de}}

\ArticleDates{Received May 22, 2019, in final form October 16, 2019; Published online October 23, 2019}

\Abstract{In this note we give a recursive formula for the derivatives of isotropic positive definite functions on the Hilbert sphere. We then use it to prove a conjecture stated by Tr\"ubner and Ziegel, which says that for a positive definite function on the Hilbert sphere to be in $C^{2\ell}([0,\pi])$, it is necessary and sufficient for its $\infty$-Schoenberg sequence to satisfy $\sum\limits_{m=0}^{\infty}a_m m^{\ell}<\infty$.}


\Classification{33B10; 33C45; 42A16; 42A82; 42C10}

\section{Introduction and main results}
In the last five years there has been a tremendous number of publications stating new results on positive definite functions on spheres, see for example \cite{Arafat2018,Beatson2014, Berg2017,Bissiri2019, Gneiting2013,Hubbert2015, Nie2018, Truebner2017,Xu2018}. Isotropic positive definite functions are used in approximation theory, where they are often referred to as spherical radial basis functions \cite{Hubbert2001,Beatson2017,Beatson2018,Castell2004} and are for example applied in geostatistics and physiology \cite{ Fornberg2015,Jaeger2016}. They are also of importance in statistics where they occur as correlation functions of homogeneous random fields on spheres \cite{Lang2015}.

A function $g\colon \mathbb{S}^{d}\times \mathbb{S}^{d} \rightarrow \mathbb{R}$ is called positive definite on the $d$-dimensional sphere
\begin{gather*}\mathbb{S}^{d}=\big \lbrace \xi \in \mathbb{R}^{d+1}\colon \Vert \xi\pdfoutput=1 \Vert_2=1 \big\rbrace \end{gather*} if it satisfies
\begin{gather*}\sum_{\xi \in \Xi}\sum_{\zeta \in \Xi} \lambda_{\xi} \lambda_{\zeta} g(\xi,\zeta) \geq 0, \end{gather*}
for any finite subset $\Xi \subset \mathbb{S}^{d}$ of distinct points on $\mathbb{S}^{d}$ and all $\lambda_{\xi}\in \mathbb{R}$. It is called strictly positive definite if the above inequality is strict unless $\lambda_{\xi}=0$ for all $\xi \in \Xi$.

Further, a function $g\colon \mathbb{S}^{d}\times \mathbb{S}^{d} \rightarrow \mathbb{R}$ is called isotropic if there exists a univariate function $\phi\colon [0,\pi]\rightarrow \mathbb{R}$ for which
\begin{gather*} g(\xi,\zeta)=\phi(\rho(\xi,\zeta)),\qquad \forall\, \xi, \zeta \in \mathbb{S}^{d},\end{gather*}
where $\rho(\xi,\zeta)=\arccos\big(\xi^T\zeta\big)$ is the geodesic distance between $\xi$ and $\zeta$.

The class of isotropic positive definite function on spheres has received more attention during the last years, even though the theory was in fact started by Schoenberg in
1942. He showed in~\cite{Schoenberg1942b} that:

\begin{theorem}[Schoenberg,~\cite{Schoenberg1942b}]\label{theorem1}
Every $\phi\colon [0,\pi]\rightarrow \mathbb{R}$ that is positive definite on $\mathbb{S}^{d}$ can be represented as
\begin{gather*} \phi(\theta)=\sum_{k=0}^\infty a_{k,d} \frac{C_{k}^{\lambda}(\cos(\theta))}{C_k^{\lambda}(1)},\qquad \theta \in[0,\pi],
\end{gather*}
where $a_{k,d} \geq 0$, for all $k$, and $\sum\limits_{k=0}^{\infty} a_{k,d} <\infty$, $\lambda:=(d-1)/2$, and the $C_{k}^{\lambda}$ are the Gegenbauer polynomials as defined in {\rm \cite[formula~(8.930)]{Gradshteyn2014}}.
\end{theorem}
The sequence $(a_{k,d})_{k\in\mathbb{N}_0}$ is referred to as a $d$-Schoenberg sequence. A criterion for the strict positive definiteness of such functions was given by Chen et al.\ in~\cite{Chen2003}.

A variety of 18 open problems on strictly and non-strictly positive definite spherical functions has been posed in the supplement material of Gneiting's article \cite{Gneiting2013}. Some of the results on these problems are described in \cite{Arafat2018,Bissiri2019,Nie2018,Truebner2017}. This note will provide some additional information to the known solution of Problem~6.

Problem 6 was concerned with the smoothness properties of the members of the class of positive definite functions on $\mathbb{S}^{d}$.
Tr\"ubner and Ziegel gave a solution to the problem in \cite{Truebner2017, Ziegel2012} and in the course of their proof stated an interesting connection between the existence of the derivative of such a function at zero and the decay of its $d$-Schoenberg sequence. The result was also described in the paper \cite[Theorem~1]{Guinness2016} by Guinness and Fuentes.
\begin{lemma}[{Tr\"ubner and Ziegel, \cite[Lemma~2.1a]{Truebner2017}}]\label{lemma1}
Let $\ell\geq1$. Suppose $\phi$ is positive definite on~$\mathbb{S}^{d}$ with $d$-Schoenberg sequence $(a_{k,d})_{k\in \mathbb{N}_0}$. Then, $\phi^{(2\ell)}(0)$ exists if and only if $\sum\limits_{k=0}^{\infty} a_{k,d}k^{2\ell}$ converges.
\end{lemma}
We show that a different connection holds for functions which are positive definite on all spheres. This function class is equivalent to the class of functions positive definite on the Hilbert sphere $\mathbb{S}^{\infty}$. For these functions Schoenberg derived a simple representation in~\cite{Schoenberg1942b}. The characterisation of strictly positive definite functions on $\mathbb{S}^{\infty}$ was later completed by Menegatto in~\cite{Menegatto1994}.

\begin{theorem}[Schoenberg,~\cite{Schoenberg1942b}]\label{theorem2}
A function $\phi$ is positive definite on $\mathbb{S}^{d}$ for all $d\geq 1$ if and only if it has the form
\begin{gather*}
\phi(\theta)=\sum_{m=0}^{\infty} a_m (\cos(\theta))^m,
\end{gather*}
where $a_m \geq0$, for all $m\in\mathbb{N}_0$, and $\sum\limits_{m=0}^{\infty} a_m <\infty$.
\end{theorem}
The series $(a_m)_{m\in\mathbb{N}_0}$ is referred to as an $\infty$-Schoenberg sequence.
In this note, we will prove the following theorem, which was shown to be true for $\ell\in\lbrace 1,2\rbrace$ in \cite{Truebner2017}, and in the process prove an interesting recursion formula for the derivatives of these positive definite functions.
\begin{theorem}\label{theorem3}
Let $\phi$ be positive definite on $\mathbb{S}^{\infty}$ with $\infty$-Schoenberg sequence $(a_m)_{m\in \mathbb{N}_0}$. Then $\phi^{(2\ell)}(0)$ exists if and only if $\sum\limits_{m=0}^{\infty}a_m m^{\ell}$ converges.
\end{theorem}
The important difference between Lemma~\ref{lemma1} and Theorem~\ref{theorem3} is the decay property of the Schoenberg sequence which is connected to the smoothness of the kernel. Theorem~\ref{theorem3} is not the limit of Lemma~\ref{lemma1} when $d\rightarrow \infty$. An explanation of the discrepancy between the cases~$\mathbb{S}^{d}$ and~$\mathbb{S}^{\infty}$ is possible using a probabilistic viewpoint. The positive definite functions are used there as covariance functions of stationary isotropic Gaussian processes on spheres. The behaviour of these processes is governed by the decay of the Schoenberg sequences. Faster decay of the sequence induces higher smoothness of the paths of the process.

As described in \cite{Bingham2019}, processes on the Hilbert sphere have very different properties from those on Euclidean spheres because the Hilbert sphere is not locally compact. Gaussian processes on~$\mathbb{S}^{\infty}$ can be discontinuous and locally deterministic, for exact definitions and explanations see~\cite{Bingham2019} and the references therein. One might expect local determinism to be a sign of greater smoothness, as for holomorphic functions in complex analysis, but it is the other way around. The processes are extremely wild. Therefore even with similarly smooth covariance functions, the process in~$\mathbb{S}^{\infty}$ is expected to have Schoenberg sequences with slower decay.

In \cite{Truebner2017} it was proven that Theorem~\ref{theorem3} is equivalent to an interesting series relation introduced as Conjecture~2.2. The conjecture contained a small typographical error in the sign of the exponent of~$2^{2j}$, in the second part of the formula. We can now prove the corrected conjecture which we state in the next lemma, and thereby prove Theorem~\ref{theorem3}.
 \begin{lemma}\label{lemma2}
 For $\ell>1$, there is a constant $c(\ell)>0$ such that, as $j\rightarrow \infty$,
 \begin{gather*}2^{-2j+1}\sum_{n=1}^j(2n)^{2\ell}\binom{2j}{j+n} \sim c(\ell)j^{\ell},\qquad 2^{-2j}\sum_{n=1}^j(2n-1)^{2\ell}\binom{2j-1}{j+n-1}\sim c(\ell)j^{\ell}.\end{gather*}
 \end{lemma}
In Section~\ref{section2} we establish necessary preliminary results which allow us to prove Lemma~\ref{lemma2}.

\section{Preliminaries}\label{section2}
First we will introduce the following lemma which might prove helpful in other areas of the discussion of positive definite functions on the Hilbert sphere.
\begin{lemma}\label{lemma3}
For $\phi(x)=\cos^j(x)$ and $j>\ell$,
\begin{gather*} \phi^{(\ell)}(x)=\sum_{\substack{n_1+n_2=\ell,\\ 0\leq n_2\leq n_1}}(-1)^{n_1} b^{j}_{n_1,n_2}\cos^{j-n_1+n_2}(x)\sin^{n_1-n_2}(x),
\end{gather*}
where the coefficients can be computed recursively by $b^{j}_{0,0}=1$,
\begin{gather*}
b^{j}_{n_1,n_2} =b^{j}_{n_1-1,n_2}(j-(n_1-1)+ n_2)+b^{j}_{n_1,n_2-1}(n_1-(n_2-1)),\qquad 0<n_2 <n_1,\\
b^{j}_{n_1,0} =b^{j}_{\ell,0}=(j-(n_1-1))b^{j}_{\ell -1,0}=\frac{j!}{(j-\ell)!},
\end{gather*}
and in the case of $\ell$ even $b^{j}_{n_2,n_2}=b^{j}_{\ell /2,\ell /2}=b^{j}_{\ell /2,\ell /2-1}$.
\end{lemma}
\begin{proof}We prove the result by induction starting with $\ell=1$.

Let $\ell=1$ and $j\in \mathbb{N}_{>1}$. Then
\begin{gather*}\phi'(x)=-j\cos^{j-1}(x)\sin(x)\end{gather*}
 thereby $b^j_{1,0}=j$. The step of the induction will be proven for odd values of~$\ell$ and even values of~$\ell$ separately.

Let $\ell$ be even, show $\ell\rightarrow \ell +1$. We assume
\begin{gather*} \phi^{(\ell)}(x)=\sum_{\substack{n_1+n_2=\ell,\\ n_2< n_1}}(-1)^{n_1} b^{j}_{n_1,n_2}\cos^{j-n_1+n_2}(x)\sin^{n_1-n_2}(x)+(-1)^{\frac{\ell}{2} }b^{j}_{\ell/2, \ell/2}\cos^j(x).\end{gather*}
Therefore for $j>\ell+1$
\begin{gather*}
 \phi^{(\ell+1)}(x)= \sum_{\substack{n_1+n_2=\ell,\\ n_2< n_1-1}} (-1)^{n_1+1} b^{j}_{n_1,n_2}(j-n_1+n_2)\cos^{j-(n_1+1)+n_2}(x)\sin^{(n_1+1)-n_2}(x)\\
\hphantom{\phi^{(\ell+1)}(x)=}{} +\sum_{\substack{n_1+n_2=\ell,\\0\leq n_2< n_1-1}}(-1)^{n_1}b^{j}_{n_1,n_2}\cos^{j-n_1+n_2+1}(x)(n_1-n_2)\sin^{n_1-(n_2+1)}(x)\\
\hphantom{\phi^{(\ell+1)}(x)=}{} +(-1)^{\ell /2+1}j\, b^j_{\ell/2, \ell/2}\cos^{j-1}(x)\sin(x) \\
 \hphantom{\phi^{(\ell+1)}(x)}{} = \sum_{\substack{\tilde{n}_1+\tilde{n}_2=\ell+1,\\ \tilde{n}_2< \tilde{n}_1-2}}(-1)^{\tilde{n}_1} b^{j}_{\tilde{n}_1-1,\tilde{n}_2}(j-(\tilde{n}_1-1)+\tilde{n}_2)\cos^{j-\tilde{n}_1+\tilde{n}_2}(x)\sin^{\tilde{n}_1-\tilde{n}_2}(x)\\
\hphantom{\phi^{(\ell+1)}(x)=}{}+\sum_{\substack{\tilde{n}_1+\tilde{n}_2=\ell+1,\\ 0 <\tilde{n}_2< \tilde{n}_1}}(-1)^{\tilde{n}_1}b^{j}_{\tilde{n}_1,\tilde{n}_2-1}\cos^{j-\tilde{n}_1+\tilde{n}_2}(x)(\tilde{n}_1-(\tilde{n}_2-1))\sin^{\tilde{n}_1-\tilde{n}_2}(x)\\
\hphantom{\phi^{(\ell+1)}(x)=}{}+(-1)^{\ell /2+1}j\, b^{j}_{\ell/2, \ell/2}\cos^{j-1}(x)\sin(x) \\
\hphantom{\phi^{(\ell+1)}(x)}{} = (-1)^{\ell+1}b^j_{\ell,0}\cos^{j-\ell-1}(x)\sin^{\ell+1}(x)(j-\ell)\\
\hphantom{\phi^{(\ell+1)}(x)=}{}+\sum_{\substack{\tilde{n}_1+\tilde{n}_2=\ell+1,\\ 0< \tilde{n}_2\leq \tilde{n}_1}} (-1)^{\tilde{n}_1}b^{j}_{\tilde{n}_1,\tilde{n}_2}\cos^{j-\tilde{n}_1+\tilde{n}_2}(x)\sin^{\tilde{n}_1-\tilde{n}_2}(x),
\end{gather*}
with $b_{\tilde{n}_1,\tilde{n}_2}$ as defined above.

Let $\ell$ be odd, show $\ell\rightarrow \ell +1$. We assume
\begin{gather*} \phi^{(\ell)}(x)=\sum_{\substack{n_1+n_2=\ell,\\ n_2< n_1}}(-1)^{n_1} b^{j}_{n_1,n_2}\cos^{j-n_1+n_2}(x)\sin^{n_1-n_2}(x).\end{gather*}
Therefore for $j>\ell+1$
\begin{gather*}
 \phi^{(\ell+1)}(x)= \sum_{\substack{n_1+n_2=\ell,\\0\leq n_2< n_1}} (-1)^{n_1+1} b^{j}_{n_1,n_2}(j-n_1+n_2)\cos^{j-(n_1+1)+n_2}(x)\sin^{(n_1+1)-n_2}(x)\\
\hphantom{\phi^{(\ell+1)}(x)=}{} +\sum_{\substack{n_1+n_2=\ell,\\0\leq n_2< n_1}} (-1)^{n_1} b^{j}_{n_1,n_2}\cos^{j-n_1+(n_2+1)}(x)(n_1-n_2)\sin^{n_1-(n_2+1)}(x)\\
\hphantom{\phi^{(\ell+1)}(x)}{} = \sum_{\substack{\tilde{n}_1+\tilde{n}_2=\ell+1,\\ 0\leq \tilde{n}_2< \tilde{n}_1-1}}(-1)^{\tilde{n}_1} b^{j}_{\tilde{n}_1-1,\tilde{n}_2}(j-(\tilde{n}_1-1)+\tilde{n}_2)\cos^{j-\tilde{n}_1+\tilde{n}_2}(x)\sin^{\tilde{n}_1-\tilde{n}_2}(x)\\
\hphantom{\phi^{(\ell+1)}(x)=}{}+\sum_{\substack{\tilde{n}_1+\tilde{n}_2=\ell+1,\\ 1 \leq \tilde{n}_2< \tilde{n}_1+1}}(-1)^{\tilde{n}_1}b^{j}_{\tilde{n}_1,\tilde{n}_2-1}\cos^{j-\tilde{n}_1+\tilde{n}_2}(x)(\tilde{n}_1-(\tilde{n}_2-1))\sin^{\tilde{n}_1-\tilde{n}_2}(x)\\
\hphantom{\phi^{(\ell+1)}(x)}{} = b^{j}_{\ell,0}(j-\ell)(-1)^{\ell+1}\cos^{j-\ell-1}(x)\sin^{\ell+1}(x)\\
\hphantom{\phi^{(\ell+1)}(x)=}{}+\sum_{\substack{\tilde{n}_1+\tilde{n}_2=\ell+1,\\ 0< \tilde{n}_2\leq \tilde{n}_1-1}}(-1)^{\tilde{n}_1} b^{j}_{\tilde{n}_1,\tilde{n}_2}\cos^{j-\tilde{n}_1+\tilde{n}_2}(x)\sin^{\tilde{n}_1-\tilde{n}_2}(x)\\
\hphantom{\phi^{(\ell+1)}(x)=}{}+ (-1)^{(\ell+1)/2}b^{j}_{(\ell+1)/2,(\ell+1)/2}\cos^j(x),
\end{gather*}
with $b^{j}_{\tilde{n}_1,\tilde{n}_2}$ as defined above.
\end{proof}

Now the behaviour of the coefficients $b^{j}_{n_1,n_2}$ for $j\rightarrow \infty$ is described.
\begin{lemma}\label{lemma4}
The coefficients $b^{j}_{n_1,n_2}$ satisfy
\begin{gather} \label{Eq:sim} b^{j}_{n_1,n_2}\sim c_{n_1,n_2}j^{n_1}, \qquad \text{for $j \rightarrow \infty$, for fixed $n_1$, $n_2$},\end{gather}
where $\sim$ means the sequences are asymptotically equivalent.
Here $c_{n_1,n_2}$ are defined recursively by
$c_{1,1}=1$, $c_{n_1,0}=1$ and for $n_1>1$, $1 \leq n_2 <n_1$
\begin{gather*}c_{n_1,n_2}=c_{n_1-1,n_2}+(n_1-n_2+1)c_{n_1,n_2-1}\end{gather*}
and $c_{n_1,n_1}=c_{n_1,n_1-1}$.
\end{lemma}
\begin{proof}We show this property by induction over the pairs $(n_1,n_2)$.
For all pairs of coefficients $(n_1,0)$ we have \begin{gather*}b^j_{n_1,0}=\frac{j!}{(j-n_1)!} \sim j^{n_1},\end{gather*} further $b^j_{1,1}=b^j_{1,0}=j\sim j^1.$

We assume \eqref{Eq:sim} holds for all combinations $(n_1,n_2)$ with $n_1\leq n'$ and $n_2\leq n_1$ and for all pairs $(n'+1,n_2)$ up to a certain $ n_2\leq n''< n'$. Then
\begin{align*}
b^{j}_{n'+1,n''+1}&=b^{j}_{n',n''+1}(j-n'+
n''+1)+b^{j}_{n'+1,n''}(n'+1-n'')\\
&\sim c_{n',n''+1} j^{n'}(j-n'+n''+1)+c_{n'+1,n''} j^{n'+1}(n'+1-n'')\\
& \sim \left( c_{n',n''+1}+c_{n'+1,n''}(n'+1-n'')\right) j^{n'+1}\\
& \sim c_{n'+1,n''+1}j^{n'+1}.
\end{align*}
For the last choice of $n''=n'$ we find
\begin{gather*} b^j_{n'+1,n'+1}=b^j_{n'+1,n'}\sim c_{n'+1,n'+1}j^{n'+1}.\tag*{\qed}
\end{gather*}\renewcommand{\qed}{}
\end{proof}

\section{Proof of Theorem~\ref{theorem3}}\label{section3}

 \begin{proof}[Proof of Lemma~\ref{lemma2}]
We use the identities of the powers of $\cos$ from Gradshteyn and Ryzhik \cite[formulas (1.320.5) and (1.320.7)]{Gradshteyn2014}
\begin{gather}
\cos^{2j}(x)=\frac{1}{2^{2j}}\left \lbrace \sum_{k=0}^{j-1}2 \binom{2j}{k}\cos(2(j-k)x)+\binom{2j}{j}\right\rbrace
\end{gather}
and
\begin{gather} \cos^{2j-1}(x)=\frac{1}{2^{2j-2}} \sum_{k=0}^{j-1} \binom{2j-1}{k}\cos((2j-2k-1)x).
\end{gather}
Differentiating each side of the above equations $2\ell$-times, using Lemma~\ref{lemma3} for the left-hand side, and evaluating the derivative at zero we find
\begin{gather*}
b_{\ell,\ell}^{2j}=\frac{1}{2^{2j}} \sum_{k=0}^{j-1}2 \binom{2j}{k}(2(j-k))^{2\ell}
\end{gather*}
and
\begin{gather*}
b_{\ell,\ell}^{2j-1}=\frac{1}{2^{2j-2}} \sum_{k=0}^{j-1} \binom{2j-1}{k}(2j-2k-1)^{2\ell}, \qquad \text{for} \quad j> \ell.
\end{gather*}
The result now follows by applying~\eqref{Eq:sim} and rearranging of the coefficients.
\end{proof}
\begin{proof}[Proof of Theorem~\ref{theorem3}]
Let $\phi$ be positive definite on $\mathbb{S}^d$ for all $d$,
\begin{gather*}\phi(\theta)=\sum_{j=0}^{\infty}a_{j,1}\cos(j\theta)=\sum_{m=0}^{\infty}a_m\cos^m(\theta),\qquad a_{j,1}\geq 0.\end{gather*}
Employing the Ziegel--Tr\"ubner result (Lemma~\ref{lemma1}), we know
that $\phi^{2\ell}(0)$ exists if and only if $\sum\limits_{k=0}^{\infty} a_{k,1}k^{2\ell}$ converges. The following relation between the Schoenberg sequences was proven, also by Ziegel and Tr\"ubner (see \cite[Proposition~5.1]{Truebner2017}):
\begin{gather*}
a_{2n,1}=\sum_{j=n}^{\infty}2^{-2j+1} a_{2j}\binom{2j}{j+n},\qquad
a_{2n-1,1}=\sum_{j=n}^{\infty}2^{-2j}a_{2j-1}\binom{2j-1}{j+n-1}.
 \end{gather*}
 This yields
\begin{align*} \sum_{n=0}^{\infty}a_{n,1}n^{2\ell}&= \sum_{n=0}^{\infty} a_{2n,1}(2n)^{2\ell}+ \sum_{n=1}^{\infty}a_{2n-1,1}(2n-1)^{2\ell}\\
 &= \sum_{n=0}^{\infty} \sum_{j=n}^{\infty}2^{-2j+1} a_{2j}\binom{2j}{j+n}(2n)^{2\ell}
 + \sum_{n=1}^{\infty}\sum_{j=n}^{\infty}2^{-2j}a_{2j-1}\binom{2j-1}{j+n-1}(2n-1)^{2\ell}.
 \end{align*}
This with $a_{n,1},a_m\geq0$ for all $m,n \in \mathbb{N}_0$ and after application of Lemma~\ref{lemma2}, proves the theo\-rem.
\end{proof}

\subsection*{Acknowledgments}

The author was a post-doctoral fellow funded by the Justus Liebig University during the development of this research. I would like to express my gratitude to Professor M.~Buhmann for his helpful comments on the paper. Thanks are also due to the anonymous referees for their thorough advice on how to improve this note.

\pdfbookmark[1]{References}{ref}
\LastPageEnding

\end{document}